\documentclass{amsart}[12pt]
\def\doctype{}

\usepackage{latexsym,amssymb,dsfont}
\usepackage{tikz}
\usetikzlibrary { decorations.pathmorphing, decorations.pathreplacing, decorations.shapes, }
\usepackage{color}
\usepackage{fancyhdr}
\usepackage{hyperref}
\hypersetup{
colorlinks=true,
citecolor=blue,
linkcolor=blue,
  urlcolor=blue
}
\newcommand\F{\mathbb{F}}
\newcommand\N{\mathbb{N}}
\newcommand\Z{\mathbb{Z}}

\newcommand\R{\mathbb{R}}

\newcommand{\comment}[1]{}

\numberwithin{equation}{section}

\fancypagestyle{titlepage}{
\fancyhead[R]{\doctype}
\fancyhead[CL]{}
\cfoot{\vspace{5pt} \thepage}
}

\let\oldsection\section
\newcommand\boldsection[1]{\oldsection{\bf #1}}
\newcommand\starsection[1]{\oldsection*{\bf #1}}
\makeatletter
\renewcommand\section{\@ifstar\starsection\boldsection}
\makeatother

\newtheoremstyle{theorem}
  {12pt}		  % space above
  {0pt}  % space below
  {\sl}  % bofy font
  {\parindent}     % ident - empty=no indent,  \parindent= paragraph indent
  {\bf}  % thm head font
  {. }    % punctuation after thm head
  { }    % space after thm head: `` ``=normal \newline=linebreak
  {}     % thm head specification
\theoremstyle{theorem}
\newtheorem{thm}{Theorem}[section]  % 1st argument is your name for it
\newtheorem{prop}[thm]{Proposition}

\newtheoremstyle{definition}
  {12pt}		  % space above
  {0pt}  % space below
  {}  % bofy font
  {\parindent}     % ident - empty=no indent,  \parindent= paragraph indent
  {\bf}  % thm head font
  {. }    % punctuation after thm head
  { }    % space after thm head: `` ``=normal \newline=linebreak
  {}     % thm head specification
\theoremstyle{definition}

\renewcommand{\proofname}{Proof}

\makeatletter
\renewenvironment{proof}[1][\proofname]{\par
  \pushQED{\qed}%
  \normalfont \partopsep=\z@skip \topsep=\z@skip
  \trivlist
  \item[\hskip\labelsep
        \scshape
    #1\@addpunct{.}]\ignorespaces
}{%
  \popQED\endtrivlist\@endpefalse
}
\makeatother

\makeatletter
\renewcommand*\@maketitle{%
  \normalfont\normalsize
  \@adminfootnotes
  \@mkboth{\@nx\shortauthors}{\@nx\shorttitle}%
  \global\topskip42\p@\relax % 5.5pc   "   "   "     "     "
  \@settitle
  \ifx\@empty\authors \else {\vskip 1em
\vtop{\centering\shortauthors\@@par}} \fi
  \ifx\@empty\@date \else {\vskip 1em \vtop{\centering\@date\@@par}}\fi % MYCHANGE
  \ifx\@empty\@dedicatory
  \else
    \baselineskip18\p@
    \vtop{\centering{\footnotesize\itshape\@dedicatory\@@par}%
      \global\dimen@i\prevdepth}\prevdepth\dimen@i
  \fi
  \@setabstract
  \normalsize
  \if@titlepage
    \newpage
  \else
    \dimen@34\p@ \advance\dimen@-\baselineskip
    \vskip\dimen@\relax
  \fi
} % end \@maketitle
\renewcommand*\@adminfootnotes{%
  \let\@makefnmark\relax  \let\@thefnmark\relax
  \ifx\@empty\@subjclass\else \@footnotetext{\@setsubjclass}\fi
  \ifx\@empty\@keywords\else \@footnotetext{\@setkeywords}\fi
  \ifx\@empty\thankses\else \@footnotetext{%
    \def\par{\let\par\@par}\@setthanks}%
  \fi
\thispagestyle{titlepage}
}
\makeatother

\makeatletter \let\c@table\c@figure \makeatother

\begin{document}

\title[]{\large The deviation from right angles in $k$-subsets of\\ points in the plane}

\author{Peter J.~Dukes}
\address{
Mathematics and Statistics,
University of Victoria, Victoria, BC, Canada
}
\email{dukes@uvic.ca}

\date{\today}

\begin{abstract}
A problem originating with Erd\H{o}s and Silverman in the 1970s asks for the minimum integer $r(k)$ such that any set of $n \ge r(k)$ points in the plane has some $k$-subset with no right angles.  The case $k=4$ has an interesting gap between the known bounds, namely $8 \le r(4) \le 10$.  Here, we consider a relaxation that quantifies the deviation from right angles.  Specifically, we study $\Gamma_k(n)$, the supremum of angles $\gamma$ such that every $n$-set of points in $\R^2$ has a $k$-subset with all angles outside of the interval $90^\circ \pm \gamma$.  We show that $4^\circ \le \Gamma_4(10) \le 9.292^\circ$.  For large $n$, the quantity $\Gamma_3(n)$ is closely related to a classical minimax angle problem pioneered by Blumenthal, Erd\H{o}s and Szekeres.  We give bounds on $\Gamma_k(n)$ for a general $k$ and large $n$.
\end{abstract}

\maketitle
\hrule

\bigskip

\section{Introduction}
\label{sec:intro}

In \cite{Erdos77}, Paul Erd\H{o}s discusses the following problem which he attributes to Silverman. {\sl What is the largest number $f(n)$ such that any set of $n$ points in the plane contain a subset of size $f(n)$, no three points of which determine a right angle?}  It is noted that an $n \times n$ grid of points shows $f(n^2) \le 2n-2$.  Some later work on this problem appears in \cite{Abbott80,AH75,Elekes91,GPRS95}.

The following `inverse' version of the problem is mentioned in \cite{FHHM02} and attributed to Andy C.~Liu: {\sl Determine the smallest positive integer $r(k)$ such that, in any collection of $r(k)$ distinct points in the plane, some $k$-subset contains no right angles.} We could not find this exact version of the problem studied elsewhere in the literature. An elegant application of the Erd\H{o}s-Szekeres theorem \cite{ES35} on monotone subsequences
gives the bound $r(k) \le (k-1)^2+1$.  To see this, let $n = (k-1)^2+1$ and choose coordinate axes for the plane such that there are no repeated $x$-coordinates nor $y$-coordinates.  By choice of $n$, there is either an increasing or decreasing $k$-term subsequence of $y$-coordinates from left to right, and hence these points cannot define any right angles.  More details for the argument are discussed in the proof of Proposition~\ref{prop:lb10} to follow.

The paper \cite{FHHM02} identifies $k=4$ as a particularly interesting first open case.  It is known that $8 \le r(4) \le 10$. The upper bound is exactly as above.  For the lower bound, consider the set of seven lattice points $\{0,1,2\}^2 \setminus \{(0,2),(2,2)\}$.  It is easy to see that any $4$-subset of these contains at least one right angle.

In this note, we introduce a parameter which captures the deviation from right angles.  Define 
$$\Gamma_k(n) = \sup\{\gamma : \forall~ S \subset \R^2, |S|=n, \exists~ T \subseteq S, |T|=k, 
|\angle ABC-90^\circ| \ge \gamma ~\forall \text{ distinct } A,B,C \in T\}.$$
For extra clarity, a lower bound $\Gamma_k(n) \ge \gamma$ asserts that every $n$-point configuration in the plane contains a $k$-subset having all angles at least $\gamma$ away from a right angle.  An upper bound $\Gamma_k(n) < \gamma'$ arises if there exists a configuration of $n$ points for which every $k$-subset has some angle within $\gamma'$ of a right angle. We note that, for fixed $k$, $\Gamma_k(n)$ is a nondecreasing function of $n$.
Our primary focus is on the case $k=4$.

From the same $7$-point configuration as above it follows that $\Gamma_4(n)=0$ for $n \le 7$. But we additionally have 
$\Gamma_4(n)=0$ for $n \in \{8,9\}$ by taking the set $S_N=\{0,N,2N\} \times \{0,1,2\}$ for very large $N$.  In any $4$-subset $T\subset S_N$, at least two points $A,B\in T$ share the same $x$-coordinate, and least one $C \in T$ has a different $x$-coordinate from $A$ and $B$. From $\|\vec{BC}\| \ge N$, $\|\vec{AB}\| \in \{1,2\}$ and $|\vec{AB} \cdot \vec{BC}| \le 4$, it follows that $|\cos \angle ABC| \le 4/N$.  This implies 
$\Gamma_4(9)$ is less than every positive real number, and hence $\Gamma_4(9) = 0$.  Of course, the same conclusion holds for $\Gamma_4(8)$.

As we see in the next section, the behavior changes at $n=10$.  We show that $\Gamma_4(10)$ is bounded away from $0$, and we exhibit a configuration of points that gives a reasonable upper bound on $\Gamma_4(10)$.  In Section 3, we note the unsurprising fact that $k$-subsets with all angles nearly degenerate are unavoidable for $n \gg k$.
That is, $\Gamma_k(n)$ tends to $90^\circ$ for large $n$, and we examine the rate of approach.  
Finally, in the last section, we briefly discuss some possible next steps and new directions for future research related to the problem.

\section{The case $k=4$ and $n=10$}
\label{sec:n10}

In this section, we give bounds on $\Gamma_4(10)$. Our lower bound essentially follows from the same monotone subsequence trick as mentioned earlier.

\begin{prop}
\label{prop:lb10}
$\Gamma_4(10) \ge 4^\circ$.
\end{prop}

\begin{proof}
Let $S \subset \R^2$ with $|S|=10$.  There are $\binom{10}{2}$ segments determined by pairs of points in $S$.  The angles of these segments relative to (say) the positive $x$-axis can be taken as elements of $[0,180^\circ)$.  Since $180^\circ/\binom{10}{2}=4^\circ$, there exists some interval $[\alpha,\alpha+4^\circ]$ containing no angle of a segment in $S$.  After applying a rotation, we may assume there are no segments in $S$ at an angle in $[0,4^\circ]$.  Suppose then that $S=\{P_1,\dots,P_{10}\}$, where $P_i=(x_i,y_i)$ and $x_1\le \cdots \le x_{10}$.  From the sequence $y_1,\dots,y_{10}$ there exists a strictly monotone subsequence $y_{i_1}, y_{i_2}, y_{i_3}, y_{i_4}$ of length $4$.  Suppose this subsequence is increasing. Then each of the angles $\angle P_{i_1}P_{i_2} P_{i_3}$, $\angle P_{i_1}P_{i_2} P_{i_4}$, $\angle P_{i_1}P_{i_3} P_{i_4}$, $\angle P_{i_2}P_{i_3} P_{i_4}$ is greater than $94^\circ$ since the $x$-coordinates are in order.  It follows that all twelve angles in this $4$-subset are outside of $90^\circ \pm 4^\circ$.
On the other hand, if the subsequence $y_{i_1}, y_{i_2}, y_{i_3}, y_{i_4}$ is decreasing, we may apply a reflection to $S$ in the line at $2^\circ$ to the positive $x$-axis.  This preserves all angles in $S$, avoids the same interval $[0,4^\circ]$ and interchanges the roles of increasing and decreasing sequences.  So we again have an increasing $4$-term subsequence in the reflected configuration and no angle within $4^\circ$ of a right angle.
\end{proof}

Our upper bound results from a computer search.  First, we implemented a simulated annealing process on a  lattice grid $\{1,\dots,N\}^2$ to get several instances of deviations around $10^\circ$.  Then, for those configurations, we applied a gradient descent perturbation to reduce the bound slightly.

\begin{prop}
\label{prop:ub10}
$\Gamma_4(10) \le 9.292^\circ$.
\end{prop}

\begin{proof}
Consider the following set of $10$ points in the plane, displayed in Figure~\ref{fig:config10}. 
$$\begin{array}{llllll}
P_0 &= (5.001213, & 67.864232)&
P_1 &= (5.00126, & 68.653515)\\
P_2 &= (12.840744, & 78.993522)&
P_3 &= (28.03804, & 64.093769)\\
P_4 &= (29.996837, & 69.918767)&
P_5 &= (32.357229, & 39.866702)\\
P_6 &= (37.862434, & 43.727817)&
P_7 &= (91.164903, & 82.745474)\\
P_8 &= (94.840088, & 68.559448)&
P_9 &= (95.819369, & 60.912308)
\end{array}$$

\begin{center}
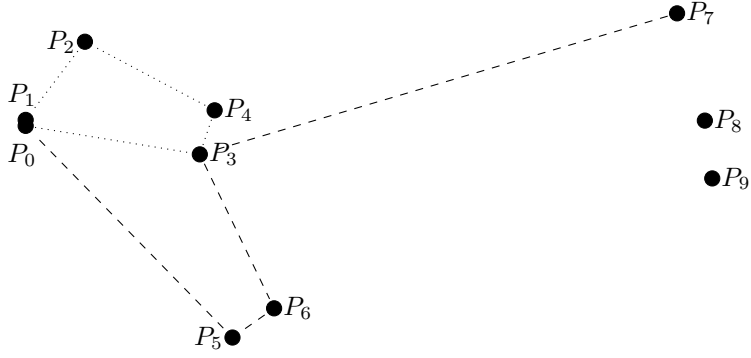
\begin{figure}[htbp]
\begin{tikzpicture}[scale=0.1]
\node (a0) at (5.001213, 67.864232) {};
\node (a1) at (5.00126, 68.653515) {};
\node (a2) at (12.840744, 78.993522) {};
\node (a3) at (28.03804, 64.093769) {};
\node (a4) at (29.996837, 69.918767) {};
\node (a5) at (32.357229, 39.866702) {};
\node (a6) at (37.862434, 43.727817) {};
\node (a7) at (91.164903, 82.745474) {};
\node (a8) at (94.840088, 68.559448) {};
\node (a9) at (95.819369, 60.912308) {};

%chains
%279
%73650
%71034215
\draw[dashed] (a7)--(a3)--(a6)--(a5)--(a0);
\draw[dotted] (a0)--(a3)--(a4)--(a2)--(a1);

\filldraw (a0) circle [radius=1];
\filldraw (a1) circle [radius=1];
\filldraw (a2) circle [radius=1];
\filldraw (a3) circle [radius=1];
\filldraw (a4) circle [radius=1];
\filldraw (a5) circle [radius=1];
\filldraw (a6) circle [radius=1];
\filldraw (a7) circle [radius=1];
\filldraw (a8) circle [radius=1];
\filldraw (a9) circle [radius=1];

\node at (4.5,64) {$P_0$};
\node at (4.5,72) {$P_1$};
\node[left] at (a2) {$P_2$};
\node[right] at (a3) {$P_3$};
\node[right] at (a4) {$P_4$};
\node[left] at (a5) {$P_5$};
\node[right] at (a6) {$P_6$};
\node[right] at (a7) {$P_7$};
\node[right] at (a8) {$P_8$};
\node[right] at (a9) {$P_9$};
\end{tikzpicture}
\caption{Arrangement of points achieving $\Gamma_4(10) \lessapprox 9.292$}
\label{fig:config10}
\end{figure}
\end{center}

\begin{center}
\begin{table}
\caption{Table of near-right angles, deviations, and fourth points of quadruples}
\label{tab:angles}
$\begin{array}{lll|lll}
\text{angle} & \gamma \lessapprox & \text{indices of fourth points} & \text{angle} & \gamma \lessapprox & \text{indices of fourth points}\\
\hline
\angle P_2 P_1 P_6  &  0.0126^\circ & 0, 3, 4, 5, 7, 8, 9 &	\angle P_4 P_6 P_8  &  6.8323^\circ & 0, 1, 2\\
\angle P_0 P_1 P_8  &  0.0566^\circ & 2, 3, 4, 5, 6, 7, 9 &	\angle P_0 P_8 P_9  &  6.8542^\circ & 2, 4, 5, 6\\
\angle P_4 P_7 P_9  &  0.1913^\circ & 0, 1, 2, 3, 5, 6, 8 &	\angle P_1 P_2 P_3  &  7.2653^\circ & 0, 4, 5, 7, 8, 9\\
\angle P_4 P_6 P_9  &  0.2006^\circ & 0, 1, 2, 3, 5, 8 &	\angle P_0 P_2 P_4  &  7.2844^\circ & 3, 7, 8, 9\\
\angle P_0 P_4 P_5  &  0.2080^\circ & 1, 2, 3, 6, 7, 8, 9 &	\angle P_5 P_4 P_7  &  7.3522^\circ & 2\\
\angle P_2 P_6 P_5  &  0.3123^\circ & 0, 3, 4, 7, 8, 9 &	\angle P_1 P_8 P_9  &  7.3575^\circ & 2, 4, 5, 6\\
\angle P_2 P_6 P_7  &  0.8479^\circ & 0, 3, 4, 8, 9 &	\angle P_1 P_3 P_4  &  7.3905^\circ & 6, 7, 8, 9\\
\angle P_2 P_0 P_6  &  1.1361^\circ & 3, 4, 8, 9 &	\angle P_1 P_5 P_7  &  7.4430^\circ & 0, 2, 6\\
\angle P_1 P_4 P_5  &  1.5933^\circ & 2, 3, 6, 7, 8, 9 &	\angle P_2 P_5 P_9  &  8.1633^\circ & 0, 1, 3, 4\\
\angle P_2 P_5 P_8  &  1.8450^\circ & 0, 1, 3, 4, 7, 9 &	\angle P_3 P_5 P_9  &  8.2384^\circ & 0, 1, 4, 6\\
\angle P_3 P_7 P_8  &  1.9362^\circ & 0, 1, 2, 4, 5, 6, 9 &	\angle P_0 P_5 P_7  &  8.2389^\circ & 2, 6\\
\angle P_3 P_6 P_8  &  2.2042^\circ & 0, 1, 2, 4, 5, 9 &	\angle P_1 P_5 P_6  &  8.4960^\circ & 0, 3, 8, 9\\
\angle P_0 P_7 P_9  &  2.2356^\circ & 1, 2, 3, 5, 6, 8 &	\angle P_4 P_8 P_9  &  8.4985^\circ & 2\\
\angle P_4 P_7 P_8  &  2.6813^\circ & 0, 1, 2, 5, 6 &	\angle P_6 P_8 P_7  &  9.0239^\circ & 5\\
\angle P_1 P_7 P_9  &  2.7460^\circ & 2, 3, 5, 6, 8 &	\angle P_6 P_9 P_8  &  9.2179^\circ & 2, 5\\
\angle P_0 P_1 P_4  &  2.9012^\circ & 2, 3, 6, 7, 9 &	\angle P_3 P_6 P_9  &  9.2371^\circ & 0, 1, 2\\
\angle P_3 P_8 P_9  &  3.4730^\circ & 0, 1, 2, 4, 5 &	\angle P_0 P_2 P_3  &  9.2726^\circ & 5, 7, 8, 9\\
\angle P_1 P_0 P_9  &  4.3740^\circ & 2, 3, 5, 6 &	\angle P_2 P_4 P_3  &  9.2903^\circ & 5, 6, 7, 8, 9\\
\angle P_3 P_7 P_9  &  4.4262^\circ & 2, 5, 6 &	\angle P_0 P_3 P_4  &  9.2913^\circ & 6, 7, 8, 9\\
\angle P_6 P_9 P_7  &  4.4810^\circ & 5, 8 &	\angle P_2 P_1 P_5  &  9.2916^\circ & 0\\
\angle P_0 P_7 P_8  &  4.7256^\circ & 2, 5, 6 &	\angle P_1 P_2 P_4  &  9.2917^\circ & 7, 8, 9\\
\angle P_6 P_4 P_7  &  4.8729^\circ & 0, 1, 3, 5 &	\angle P_2 P_7 P_9  &  9.2919^\circ & 8\\
\angle P_1 P_7 P_8  &  5.2360^\circ & 2, 5, 6 &	\angle P_0 P_5 P_6  &  9.2919^\circ & 3, 8, 9\\
\angle P_5 P_4 P_8  &  5.6920^\circ & 3, 6, 9 &	\angle P_0 P_1 P_7  &  9.2919^\circ & 2, 3, 6\\
\angle P_5 P_3 P_8  &  6.2840^\circ & 0, 1 &	\angle P_6 P_3 P_7  &  9.2919^\circ & 0, 1\\
\angle P_5 P_9 P_7  &  6.3125^\circ & 2, 8 &	\angle P_1 P_0 P_3  &  9.2919^\circ & 5, 6\\
\angle P_5 P_3 P_7  &  6.3521^\circ & 0, 1, 2, 4, 6 &	\angle P_3 P_6 P_5  &  9.2919^\circ & 4\\
\end{array}$
\end{table}
\end{center}

The nearest angle to a right angle within every $4$-subset is displayed in Table~\ref{tab:angles}.  For a compact presentation, the table lists several angles with deviations from right angles under $9.292^\circ$.  
The third column gives the indices of fourth points which form $4$-subsets with those of the angle and for which that angle achieves the minimum deviation from $90^\circ$.  In total, all $\binom{10}{4}$ quadruples of points contain one of the listed angles.
\end{proof}

It unfortunate that our configuration $P_0,P_1,\dots,P_9$ lacks any obvious structure or visual appeal.  
On the other hand, it is interesting that many of the angles near maximum deviation arise from chains of consecutive segments. Two such chains are indicated in Figure~\ref{fig:config10}.  This possibly hinders  the ability to make na\"ive local improvements through small perturbations.
  
\section{Bounds for general $k$ and large $n$}
\label{sec:large}

In this section, we consider $\Gamma_k(n)$ for a general $k$ and large $n$.  
As a warm-up, we first sketch a simple argument to show that $\Gamma_k(n)$ tends to $90^\circ$.  Fix a positive integer $m$ and create $m$ equal-sized angular bins of width $180^\circ/m$.  Take $n \ge R_m(k)$, the $m$-color Ramsey number guaranteeing a monochromatic $k$-clique.  Color the $\binom{n}{2}$ segments between points according to their angle relative to (say) the positive $x$-axis.  By choice of $n$, there exists a $k$-subset of points with all angles between segments in the same bin.  This shows $90^\circ-\Gamma_k(n) \le 180^\circ/m$ for sufficiently large $n$. 

The above shows that, for large $n$, some $k$-subset will have all triangles nearly degenerate.  For this reason, it is convenient to introduce $\Delta_k(n):=90^\circ-\Gamma_k(n)$, where a value near $0$ corresponds with degenerate $k$-subsets.  
We give upper and lower bounds, showing that $\Delta_k(n)$ is roughly on the order of $\log(k)/\log(n)$.

With only a little bit of work, we can improve on the size of $n$ in the upper bound, replacing the Ramsey number with an exponential in $m$.

\begin{prop}
\label{prop:lb-general}
For a constant $C$ and $n \ge k \ge 3$, we have 
$\Delta_k(n) \le C\log(k)/\log(n)$.
\end{prop}

\begin{proof}
Let $m \in \Z$, $m \ge 2$.
Consider a set of $n$ points $P_1,\dots,P_n$ in the plane, where $n \ge (k-1)^m+1$.  We may assume they are listed in order of increasing $x$-coordinate.  As above, partition the angles of segments $P_iP_j$, $i<j$, into $m$ bins, each of width $180^\circ/m$.  We claim there exists a $k$-term subsequence $P_{i_1},P_{i_2},\dots, P_{i_k}$  with segments $P_{i_1}P_{i_2}, \dots, P_{i_{k-1}}P_{i_k}$ in the same bin. To see this, give each point $P_i$ a label $\ell(P_i) \in \N^m$, where the $b$th coordinate of $\ell(P_i)$ is the longest chain of a subsequence of points ending at $P_i$ in which all segments between consecutive elements belong to bin $b$.  If any such coordinate is $k$ or more, we are done.  So we may assume $\ell(P_i) \in \{1,\dots,k-1\}^m$ for each $i$.  By choice of $n$, some two points, say $P_i,P_j$ with $i<j$, have the same label.  The segment $P_iP_j$ belongs to some bin, say the $b$th, but this gives the contradiction $\ell(P_j)_b > \ell(P_i)_b$.  Now, with segments $P_{i_1}P_{i_2}, \dots P_{i_{k-1}}P_{i_k}$ in the same angular bin, it follows from convexity that all segments in our subsequence are in this bin.  Similar to the proof of Proposition~\ref{prop:lb10}, all $3\binom{k}{3}$ angles are within $180^\circ/m$ of either $0$ or $180^\circ$.  This shows 
$\Delta_k(n) \le 180^\circ/m \le C\log k/\log n$ for a constant $C$.
\end{proof}

Next, we sketch a simple construction for a lower bound on $\Delta_k(n)$.  Let $S$ be a set of $n$ equally-spaced points on a circle with center $O$.  For any $k$-subset $T \subset S$, there exist two element $A,C \in T$ for which $\angle AOC \ge \frac{(k-1)}{n}360^\circ$.  Then, choosing some $B \in T$ gives $\frac{k-1}{n}180^\circ \le \angle ABC \le (1-\frac{k-1}{n})180^\circ$, and this shows $\Delta_k(n) \ge \frac{k-1}{n} 180^\circ$.  

To improve on this, we turn to a construction of Szekeres \cite{Szekeres41}.  
When converted to our notation, this construction originally gave a nearly optimal lower bound on $\Delta_3(n)$.  We adapt the method to do a bit better for $\Delta_k(n)$ and essentially match the upper bound given earlier in Proposition~\ref{prop:lb-general}.

\begin{prop}
\label{prop:ub-general}
For a constant $C'$ and $n \ge k \ge 3$, we have 
$\Delta_k(n) \ge C' \log(k)/\log(n)$.
\end{prop}

\begin{proof}
We review the construction of Szekeres from \cite{Szekeres41}.  First, let $n=2^t$ for a positive integer $t$ and index points $P_{\mathbf{v}}$ by binary vectors $\mathbf{v}=(v_0,v_1,\cdots, v_{t-1}) \in \{0,1\}^t$.  In the complex plane, we let
$$P_{\mathbf{v}}:=\sum_{j=0}^{t-1} v_j R^{j} e^{i\pi j/t}.$$
Consider any $k$-subset $K \subset \{0,1\}^t$. Let $J$ denote the maximum index in which any two vectors of $K$ disagree, and say $\mathbf{v}$, $\mathbf{v}' \in K$ with $v_J \neq v'_J$. Without loss of generality, $v_J=1$ and there is a subset $K' \subset K$ with $|K'| \ge \lfloor k/2 \rfloor$ such that $w_J=0$ for all $\mathbf{w} \in K'$.  By the pigeonhole principle, some two elements of $K'$ agree in at least $\log_2 |K'|$ bits to the left of $J$.  So take binary vectors 
$\mathbf{u},\mathbf{w} \in K'$ with this property, agreeing (say) in positions $H+1,\dots,J$.  Consider the triangle $\triangle P_\mathbf{u}P_\mathbf{v}P_\mathbf{w}$.  For large $R$, the approximate directions formed by the sides $P_\mathbf{u}P_\mathbf{v}$, $P_\mathbf{v}P_\mathbf{w}$ and $P_\mathbf{u}P_\mathbf{w}$ are, respectively, $e^{i \pi J/t}$, $-e^{i \pi J/t}$, and $\pm e^{i \pi H/t}$. But $J-H$ is at least order $\log(k)$ and $t$ is at most order $\log(n)$. It follows that, in $\triangle P_\mathbf{u}P_\mathbf{v}P_\mathbf{w}$, the angle at
either $P_\mathbf{u}$ or $P_\mathbf{w}$ (depending on sign) is at most $180^\circ-C'\log(k)/\log(n)$ for a constant $C'$. 

Finally, for $2^{t-1} < n < 2^t$ we note that the construction can be applied with $2^t \le 2n$ points, and the bound only changes by a constant multiple.
\end{proof}

Constructions toward a lower bound on $\Delta_3(n)$ in the range $2^{t-1} < n < 2^t$ are an interesting topic in their own right.  Sendov \cite{Sendov95} obtains some partial results on this problem.

\section{Conclusion and Discussion}
\label{sec:conclusion}

We introduced a quantitative refinement of Erd\H{o}s and Silverman's problem, making some initial observations on $\Gamma_k(n)$, the angular deviation from right angles in $k$-subsets among $n$ points in the plane.  We obtained a rate of approach to degeneracy on the order of $\log(k)/\log(n)$.  There remains some work on the constants in this analysis.  Looking forward from our result that $\Gamma_4(10) \in [4^\circ,9.292^\circ)$, it seems fairly likely that both ends can be tightened. Improvements on either side, including a structural classification of the extremal configurations, would be a worthwhile next step in the research.

Computationally-generated configurations for the next smallest values $n=11,12,\dots$ might give valuable insight on both structure and the approach to degeneracy.  After only a few annealing runs on a $100 \times 100$ lattice, we obtained $\Gamma_4(11) < 16^\circ$.  Points achieving this are $$
(15, 14),
 (21, 10),
 (31, 5),
 (36, 79),
 (36, 80),
 (61, 78),
 (62, 73),
 (74, 61),
 (85, 90),
 (89, 85),
 (93, 72).$$
We also found perturbations via gradient descent that push the upper bound on $\Gamma_4(11)$ slightly under $15^\circ$.

We are hopeful that $r(k)$ in the original problem on right angles can be better understood through a careful analysis of $\Gamma_k(n)$ for $n$ on the order of $k^2$.  To this end, a variant of our problem that imposes some minimum relative separation between points may be useful to explore.  
Our construction showing $\Gamma_4(9)=0$ relied on clusters of essentially identical points in the limit.  Although our construction upper-bounding $\Gamma_4(10)$ has two points $P_0,P_1$ quite close together, many computer-generated examples at a slightly worse bound of around $10^\circ$ had larger separation between points.

A Tur\'an-type result on the maximum number of near right-angled triangles on $n$ points in the plane is given in \cite[Theorem 1.2(a)]{BCD25}. This work considers other triangles, including equilateral and isosceles.  A Ramsey version in these other cases may be interesting to explore.

Finally, it is reasonable to ask about similar problems on point sets in $\R^d$, or some other space such as $\F_q^d$.  Angles in the latter setting essentially correspond with Hamming distance. The paper \cite{BMPS23} studies right-angle-free subsets in finite fields and the connection to coding theory.

\section*{Acknowledgements}

Research of Peter Dukes is 
supported by NSERC Discovery Grant RGPIN-312595-2023.  In preparing this work, Claude (Anthropic) was used for assistance with coding and mathematical discussion. Thanks are also given to Felix Clemen, who identified the Szekeres construction as a good candidate for this problem.

%%%%%%%%%%%%%%%%

\end{document}